\documentclass{amsart}

\usepackage{amssymb}
\usepackage{amsthm}
\usepackage{amsmath, a4wide}
\usepackage{amsbsy}
\usepackage[all]{xy}
\usepackage{bm}
\usepackage{hyperref}

\newcommand \codim {\operatorname{codim}}
\newcommand\spa{\operatorname{span}}
\newcommand\R{\mathbb{R}}

\usepackage{amssymb, amsmath, amsthm}

\begin{document}

\newtheorem{thm}{Theorem}
\newtheorem{prop}{Proposition}
\newtheorem*{conjecture}{Conjecture}
\newtheorem{lemma}{Lemma}

\title[Convolution of Singular Measures and Brascamp-Lieb Inequalities]{Convolution Estimates for Singular Measures and\\ Some Global Nonlinear Brascamp-Lieb Inequalities}

\author{Herbert Koch}
\address{Mathematisches Institut, Universit\"at Bonn, Endenicher Allee 60, 53115 Bonn, Germany}
\email{koch@math.uni-bonn.de}

\author{Stefan Steinerberger}
\address{Department of Mathematics, Yale University, 10 Hillhouse Avenue, CT 06511, USA}
\email{stefan.steinerberger@yale.edu}

\begin{abstract}
We give a $L^2\times L^2 \rightarrow L^2$ convolution estimate for singular measures supported on transversal hypersurfaces in $\mathbb{R}^n$, which improves earlier results of Bejenaru, Herr \& Tataru
as well as Bejenaru \& Herr. The arising quantities are relevant in the study of the validity of bilinear 
estimates for dispersive partial differential equations. We also prove a class of global, nonlinear Brascamp-Lieb inequalities with explicit constants in the same spirit. 
\end{abstract}

\maketitle

\section{Introduction}
\subsection{Loomis-Whitney.} The classical Loomis-Whitney inequality \cite{LW} bounds the $n-$dimensional volume of an open subset of $\mathbb{R}^n$
in terms of the size of its projections onto $(n-1)$-dimensions, one formulation is as follows: if the projections $\pi_{j}:\mathbb{R}^n \rightarrow \mathbb{R}^{n-1}$ are given by omitting
the $j-$th component $\pi_{j}(x) = (x_{1}, \dots, x_{j-1}, x_{j+1}, \dots, x_{n})$, then
$$ \int_{\mathbb{R}^{n}}{ f_{1}(\pi_{1}(x))\dots f_{n}(\pi_{n}(x)) dx} \leq \|f_{1}\|_{L^{n-1}(\mathbb{R}^{n-1})}\dots \|f_{n}\|_{L^{n-1}(\mathbb{R}^{n-1})}$$
for all $f_{j} \in L^{n-1}(\mathbb{R}^{n-1})$. The isoperimetric inequality is an immediate consequence: given a bounded domain $\Omega \subset \mathbb{R}^n$ and letting $f_{j}$ be the characteristic function
of $\pi_{j}(\Omega)$, then any reasonable definition of surface measure satisfies
$$ \|f_{j}\|_{L^{n-1}(\mathbb{R}^n)}^{n-1} \leq |\partial \Omega | $$ 
and thus the Loomis-Whitney inequality implies
$$ |\Omega| = \int_{\mathbb{R}^{n}}{ f_{1}(\pi_{1}(x))\dots f_{n}(\pi_{n}(x)) dx} \leq \|f_{1}\|_{L^{n-1}(\mathbb{R}^{n-1})} \dots \|f_{n}\|_{L^{n-1}(\mathbb{R}^{n-1})} \leq |\partial \Omega|^{\frac{n}{n-1}}.$$
The proof of the Loomis-Whitney inequality combines an elementary combinatorial setup
with induction on dimension, where each induction step is an
application of H\"olders inequality.  A non-linear and local form of
this result was given by Bennett, Carbery \& Wright \cite{BCW}; if the
$\pi_{j}$ are submersions with the kernels of $d\pi_{j}$ spanning the
whole of $\mathbb{R}^n$, then the nonlinear Loomis-Whitney inequality
still holds locally (with appropriate cut-off functions and
constants).

\subsection{Convolution inequalities.} A variant in $\mathbb{R}^3$
arises in the Fourier analysis of nonlinear dispersive equations and
is due to Bejenaru, Herr \& Tataru \cite{BHT}.  Let $\Sigma_{1},
\Sigma_{2}, \Sigma_{3} \subset \mathbb{R}^3$ be three surfaces with
$C^{1+\alpha}$ regularity, bounded in diameter by 1 and uniformly
transversal in the sense that for any three points $x_{i} \in
\Sigma_{i}$ the associated normal vectors $\nu_{i}$ satisfy
$$ |\det(\nu_{1},\nu_{2},\nu_{3})| \geq \frac{1}{2}.$$
Writing $\mu_{\Sigma_i}$ for the two-dimensional Hausdorff measure $\mathcal{H}^2$ restricted to $\Sigma_i$, we identify $f_{1} \in L^2(\Sigma_{1},\mu_{\Sigma_1})$ with the distribution
$$ f_1(\psi) = \int_{\Sigma_1}{f_1(x)\psi(x) d\mu_{\Sigma_1}}(x) \quad \psi \in \mathcal{D}(\R^3)$$
and analogously for $f_2 \in L^2(\Sigma_2, \mu_{\Sigma_2})$, which allows us to define the convolution as
$$ (f_1*f_2)(\psi) = \int_{\Sigma_1}{\int_{\Sigma_2}{f_1(x)f_2(y)\psi(x+y)d\mu_{\Sigma_1}(x)}d\mu_{\Sigma_2}(y)}.$$
Here and below, we can always assume all arising functions to be nonnegative; if the inequalities hold for nonnegative functions,
then they also hold for all measurable functions.

Thickening the surfaces one sees by the coarea formula that 
\[ (f_1*f_2)(z) =   \int_{\Sigma_1\cap (z-\Sigma_2)}  \gamma_3^{-1}(x,z-x)   f_1(x)f_2(z-x) d\mathcal{H}^{1}  \] 
where $\gamma_3(x,y)$ is  the cosine  of the  angle between $\nu_1(x)$ and 
$\nu_2(y)$ for $x \in \Sigma_1$ and $y \in \Sigma_2$. Bejenaru, Herr \& Tataru then show
that the restriction is well-defined in
$L^2(\Sigma_3,d\mu_{\Sigma_3})$ and that it satisfies
$$ \|f_{1}*f_{2}\|_{L^2(\Sigma_{3})} \lesssim \|f_{1}\|_{L^2(\Sigma_{1})} \|f_{2}\|_{L^2(\Sigma_{2})}.$$

The behavior under linear transformations yields additional
information: one should be able to weaken the assumption on the
transversality to $ |\det(\nu_{1},\nu_{2},\nu_{3})| \geq \gamma > 0$ at the
cost of increasing the implicit constant by a factor of order
$\gamma^{-\frac12}$ and, indeed, this is done in \cite{BHT} at the
cost of assuming certain conditions on diameter, H\"older exponent and
H\"older norm. A dual formulation is achieved by introducing a weight
function $f_3 \in L^2(\Sigma_{3})$ and rewriting
$$ \int_{\Sigma_3}{(f_1*f_2)(x)f_3(-x)dx} = (f_1*f_2*f_3)(0).$$
Considering thickened surfaces $\Sigma_{i}^{*} =
\Sigma_{i}+B(0,\varepsilon)$, where $B(0,\varepsilon)$ is the ball of
radius $\varepsilon$ in $\mathbb{R}^3$, and assuming $f_1 \in
L^2(\Sigma_{1}^{*})$, $f_2 \in L^2(\Sigma_{2}^{*})$ and $f_3 \in
L^2(\Sigma_{3}^{*})$, then ignoring all underlying geometry implies
with H\"older that
\begin{align*}
 |(f_1*f_2*f_3)(0)|  &= \left|\int_{\mathbb{R}^3} \mathcal{F}(f_1*f_2*f_3)(x)dx \right| \\
 &= \left| \int_{\mathbb{R}^3}{\hat{f_1}(\xi)\hat{f_2}(\xi)\hat{f_3}(\xi)d\xi}\right| \leq \|\hat{f_1}\|_{L^3}\|\hat{f_2}\|_{L^3}\|\hat{f_3}\|_{L^3}
\end{align*}
where $\mathcal{F}$ and $\hat{}$ denote the Fourier transform. 
One way of looking at the Bejenaru-Herr-Tataru statement is that for
$\varepsilon \rightarrow 0$ the transversal structure of the Fourier
supports implies additional cancellation and allows us to conclude
$$ \int_{\mathbb{R}^3}{\hat{f_1}(\xi)\hat{f_2}(\xi)\hat{f_3}(\xi)dx} \lesssim \|\hat{f_1}\|_{L^2}\|\hat{f_2}\|_{L^2}\|\hat{f_3}\|_{L^2}.$$

The proof given by Bejenaru, Herr \& Tataru is quite remarkable: it
uses induction on scales \`{a} la Wolff and has inspired recent work
by Bennett \& Bez \cite{BB} on Brascamp-Lieb inequalities.  The work
of Bennett \& Bez was then used by Bejenaru \& Herr \cite{BH} to extend
\cite{BHT} to arbitrary dimensions under the natural scaling condition
of the codimensions adding up to the space dimension. All three 
papers treat the nonlinearity in a perturbative fashion.

\subsection{Applicability.} Results of this type are related to
(multilinear) restriction problems and a bilinear estimates for
partial differential equations of dispersive type, where the type of
nontrivial interaction of two characteristic hypersurfaces $\Sigma_1 +
\Sigma_2$ with a third hypersurface $\Sigma_3$ determines whether a
bilinear estimate is available. Bejenaru \& Herr \cite{BH}, for
example, use their generalized version of the Bejenaru-Herr-Tataru
result to obtain locally well-posedness for the 3D Zakharov system in
the full subcritical regime.

\section{Statement of results} 
\subsection{The simplest case.} We are interested in general
convolution inequalities for curved submanifolds of
$\mathbb{R}^n$. The simplest case, taken from \cite{BHT}, is given by
three transversal hyperplanes in $\mathbb{R}^3$ equipped with the
two-dimensional Hausdorff measure $\mathcal{H}^2$
\begin{align*}
\Sigma_1 &= \left\{(x,y,z) \in \R^3: x=0\right\}  \\
 \Sigma_2 &= \left\{(x,y,z) \in \R^3: y=0\right\}   \\
\Sigma_3 &= \left\{(x,y,z) \in \R^3: z=0\right\}
\end{align*}
and smooth functions $f \in L^2(\Sigma_1,\mathcal{H}^2),g \in
L^2(\Sigma_2,\mathcal{H}^2)$ and $h \in
L^2(\Sigma_3,\mathcal{H}^2)$. The convolution can be written down in
an explicit fashion
$$ (f_{}*g_{})(x,y,z) := \int{f(y,z')g(x,z-z')dz'}$$
and then duality yields that the estimate
$$ \|f * g\|_{L^2(\Sigma_3,\mathcal{H}^2)} \leq \|f\|_{L^2(\Sigma_1,\mathcal{H}^2)}\|g\|_{L^2(\Sigma_2,\mathcal{H}^2)}$$
is equivalent to the estimate
$$ \left| \int{f_{}(y,z) g(x,-z) h(x,y) dx dy dz} \right| \leq \|f\|_{L^2(\Sigma_1,\mathcal{H}^2)}\|g\|_{L^2(\Sigma_2,\mathcal{H}^2)}\|h\|_{L^2(\Sigma_3,\mathcal{H}^2)},$$
which in itself is simply the three-dimensional Loomis-Whitney
inequality. Note that the affine structure of the hyperplanes is
crucial for the proof to work as it allows for an explicit parametrization
of the integration fibers: in particular, this proof is not stable under small
perturbations of the underlying surfaces.
%\begin{figure}[h]
%\centering
%\begin{tikzpicture}
%\begin{scope}[line width=2pt]
%\draw (0,0) .. controls (2,-0.1) and (3,-0.3) .. (4,-1);
%\draw (4,-1) .. controls (5,0) and (5.3,0) .. (6,1);
%\draw (6,1) .. controls (4,1.8) and (3,1.7) .. (2,1.6);
%\draw (0,0) .. controls (1,0.7) and (1,0.7) .. (2,1.6);
%\end{scope}
%\end{tikzpicture}
%\caption{A perturbed plane.} 
%\end{figure}

%\begin{figure}[h!]
%\begin{center}
%\includegraphics[width = 6cm]{skizze4.jpg}
%\end{center}
%\caption{Three nonplanar transversal surfaces in $\R^3$.}
%\end{figure}
\subsection{Setup.} Our general setup is as follows. Let $\Sigma_i$ ($i =1,2,3$) be $n_i-$dimensional Lipschitz manifolds 
in $\mathbb{R}^n$ with codimensions adding up to the space dimension, i.e.
\[   \sum_{i=1}^3 (n-n_i) = n \qquad \mbox{or} \qquad n_1 + n_2 + n_3 = 2n.  \]
This condition will be necessitated by scaling.
Let $\mu_{\Sigma_i}$ be the $n_i-$dimensional Hausdorff measure restricted to $\Sigma_i$. We associate with every $f \in L^1_{loc}( \Sigma_i, \mu_i)$ the 
signed measure $f\mu_{\Sigma_i}$.  For every $i \in \left\{1,2,3\right\}$, there exists 
an orthonormal basis $\nu_{i,j}$, $1\le j \le n-n_i$, of normal vectors at every point of $\Sigma_i$ almost everywhere. Since $(n-n_1) + (n-n_2) +(n-n_3) =n$, this means that for every $(x,y,z) \in \Sigma_1 \times
\Sigma_2 \times \Sigma_3$, we can define the matrix $N(x,y,z)$ by collecting the $n$
vectors $\nu_{i,j}$ as columns,
\[ N(x,y,z) = \left( \nu_{i.j} \right) \]
The natural measure of transversality $\gamma$ is then defined by
$$ \Sigma_{1} \times \Sigma_{2} \times \Sigma_{3} \ni (x,y,z) \to \gamma(x,y,z) = |\det(N(x,y,z))|.$$
We call the surfaces {\em transversal} if this local measure of transversality is uniformly bounded
from below $\gamma \geq \gamma_0 > 0$ for all $(x,y,z) \in \Sigma_1 \times \Sigma_2 \times \Sigma_3$, whenever all three normal vectors are defined.

\subsection{A convolution inequality.} Our first result is a  global version of \cite{BH} without the requirements on H\"older continuity, H\"older norms or bounds on the diameter -- furthermore we are able
to give an explicit constant.
\begin{thm}[Convolution inequality] \label{convolution} 
Suppose that the $\Sigma_i$ are given as the graphs of Lipschitz functions and that 
$$\gamma(x,y,z) \ge \gamma_0.$$ 
 Then, for all $f_1 \in L^2(\Sigma_1), f_2 \in L^2(\Sigma_2)$, we have
$$  \|f_{1}*f_{2}\|_{L^2(\Sigma_{3})}  \le  \gamma_0^{-\frac{3}{2}} \Vert f_1 \Vert_{L^2(\Sigma_1)} \Vert f_2 \Vert_{L^2(\Sigma_2)}.$$
\end{thm}
The constant does not have the form $\gamma_0^{-\frac{1}{2}}$ of the
case of linear hypersurfaces; this is due to the fact that in the
proof we are forced to use rough estimates to control a nonlocal
interaction between locally defined linear maps -- studying the factor
explicitely can give slightly better bounds in many cases. 

The following questions are natural but, to the best of our knowledge, open. 

\begin{enumerate} 
\item Suppose that 
\[ 
\inf\{ \gamma(x,y,z) : x\in \Sigma_1, y\in \Sigma_2, z \in \Sigma_3,
x+y=z \} \geq \eta_0 > 0. 
\]
Is there a constant $C=C(\eta_0)$ so that the convolution estimate holds with this constant? We do not even know the answer if $\eta_0$ is larger than $1/2$. 
\item Can one replace the factor $\gamma_0^{-3/2}$ by $C\gamma_0^{-1/2}$? 
Our proof shows  this to be the case if the submanifolds are controlled Lipschitz perturbations 
of linear subspaces. 
\end{enumerate} 

Our proof gives the following intermediate 
 refinement of Theorem 1. If we assume that
\[ 
\inf\{ \gamma(x,y,z) : x,\tilde x \in \Sigma_1, y,\tilde y \in
\Sigma_2, z \in \Sigma_3, x+\tilde y=\tilde x + y =z \} =: \gamma_0 > 0,
\]
then 
$$  \|f_{1}*f_{2}\|_{L^2(\Sigma_{3})}  \le \gamma_0^{-\frac{3}{2}} \Vert f_1 \Vert_{L^2(\Sigma_1)} \Vert f_2 \Vert_{L^2(\Sigma_2)}.
$$

\subsection{Nonlinear Brascamp-Lieb inequalities.} Our approach to the convolution problem is flexible
enough to allow us to deduce new global, nonlinear Brascamp-Lieb inequalities. We formulate the Brascamp-Lieb inequalities for the space $\R^n$ and three nonlinear mappings $\phi_i: \R^n \rightarrow \R^{n_i}$,
where again
$$n_1 + n_2 + n_3 = 2n.$$
Note that the problem now lies in the nonlinear fiber structure
induced by preimages $\phi_i^{-1}$.  We merely assume the $\phi_i$ to
be $C^1-$submersions (i.e. $D\phi_i$ has rank $n_i$). Let $N_i(x)$
be the null space of $D\phi_i(x)$ of dimension $n-n_i$.  We recall
that $$(n-n_1)+(n-n_2)+(n-n_3)=n$$ and assume that the nullspaces span $\R^n$
at every point.  By an abuse of notation we identify $N_i$ with a matrix 
having an orthonormal basis as columns. We introduce a measure of
transversality via
\[ \gamma_0 = \inf_{z \in \R^{n_3}} \inf_{\{x,y: \phi_3(x)= \phi_3(y) = z\} } 
  \sup_{O \in O(n) : ON_3(x) = N_3(y) } 
|\det (ON_1(x), N_2(y), N_3(y) ) |. 
\]  
This notion is inspired by the proof for the convolution, where such a definition
is required in order for the Brascamp-Lieb case to behave in an analogous fashion
as the convolution. To shorten the notation for the coarea formula we will henceforth write
for a matrix $A$
\[ |A| = (\det A A^T)^{1/2} \] 
Furthermore, we will use $\sigma_j(x)$ to denote the singular values of $D\phi_3(x)$, 
i.e. the square roots of the eigenvalues of $D\phi_3 D\phi_3^T$, ordered by their 
size, $\sigma_1 \le \sigma_2 \le \dots \le \sigma_{n_3}$.
Using this notation, we have in particular that
\[|D\phi_3(x)| = \prod_{j=1}^{n_3}\sigma_j \] 
and obtain for the operator norm of the linear mapping that
\[ \Vert D\phi_3(x)\Vert_{\ell^2 \rightarrow \ell^2} = \sigma_{n_3}. \]  
Since the $D\phi_i$ are assumed to always have maximal rank, we necessarily have $\sigma_1>0$. 
Introducing additional notation, we set
\[ \rho_{1}(x) =    |D\phi_3(x)|^{-\frac{n-n_2}{n_3}}  \prod_{j=n-n_2+1}^{n_3} \sigma_j \] 
as well as
\[ \rho_{2}(y)=  |D\phi_3(x)|^{-\frac{n-n_1}{n_3}} \prod_{j=n-n_1+1}^{n_3} \sigma_j. \] 
We observe that $n-n_1+n-n_2 = n_3 $.  Let 
\[ \rho=\sup_{z \in \R^{n_3}} \sup_{\{x,y: \phi_3(x)= \phi_3(y) = z \} } 
    \rho_1(x) \rho_2(y).  
\]  

\begin{thm}[Loomis-Whitney/Brascamp-Lieb] Under these assumptions, if 
$$\rho<\infty  \quad \mbox{and} \quad \gamma_0>0,$$ 
then for all $f_i \in L^2(\mathbb{R}^{n_i})$
$$   \int_{\R^n}  \left(\prod_{i=1}^3 |D\phi_1||D\phi_2||D\phi_3|\right)^{\frac12}
 f_1(\phi_1(x)) f_2(\phi_2(x)) f_3(\phi_3(x)) dx \le 
\sqrt{\frac{\rho}{\gamma_0}} \Vert f_1 \Vert_{L^2} \Vert f_2 \Vert_{L^2} \Vert f_3 \Vert_{L^2}.$$
\end{thm}   

The geometric weights arise as compensation of two possible kinds of
symmetries: the geometric factor under the integral compensates for
the possibility of coordinate changes in the image. We would like to
obtain a formulation that is invariant under composing any of the
three mappings $\phi_i$ with an invertible linear mapping $A:\R^{n_i}
\rightarrow \R^{n_i}$, which introduces a Gramian determinant as
Jacobian -- this can happen in each image space independently which
implies the necessity of the product structure of the weight. We do
not quite achieve this desired invariance, since $\rho$ changes under
diffeomorphisms of $\R^{n_3}$. This observation leads to a trivial
improvement (minimizing over all possible diffeomorphisms) which we do
not explicitely formulate. In addition the transversality of the
kernels of $D\phi_i(x)$ has to enter the estimate. As for the
convolution our condition is not entirely local and the question
whether purely local transversality measures suffice, remains open.

\section{Proof of Linear Convolution inequalities}

\subsection{Outline.} The purpose of this section is to prove Theorem
1 in the case of the surfaces $\Sigma_i$ being linear subspaces and Theorem 2 
for linear maps $\phi_i$. These
special cases are known: it suffices to use the previously outlined
approach via the Loomis-Whitney inequality and combine it with a
change of variables. However, as outlined before, that proof is not
stable. We will give a new proof of this special case; that proof will
then be stable enough to handle the nonlinear setting as well. We
start with a description of the underlying geometry, which translates
into statements about parallelepipeds. Geometric properties about
these parallelepipeds will determine factors coming  from the
transformation formula.

\subsection{Basic facts about parallelepipeds.} Let $H_1, H_2$ and
$H_3$ be subspaces of $\R^n$ such that the direct sum of the
orthogonal spaces yield $\R^n$, i.e.
$$\R^n = H_{1}^{\perp} \oplus H_{2}^{\perp} \oplus  H_{3}^{\perp}.$$
 Denote $n_i = \dim(H_i)$
and thus $\codim(H_i) = \dim(H_i^{\perp}) = n-n_i$. The vectors $(v^i_j)_{i=1,2,3,j=1,\dots ,n-n_i}$ are chosen such
that $(v^i_j)_{j=1,\dots ,n-n_i}$ form an orthonormal basis of $H_i^{\perp}$. 
There is a unique dual  basis  $(w^i_j)_{i=1,2,3,j=1,\dots ,n-n_i}$ defined by 
the  relations
$$ \left\langle v^i_j, w^k_l \right\rangle = \delta_{ik}\delta_{jl}.$$
We write $V_i$ for the $n \times (n-n_i)$ matrix containing $(v^i_j)_{j=1,\dots ,n-n_i}$ and
$W_i$ for the matrix of the same size containing $(w^i_j)_{j=1,\dots ,n-n_i}$. Our condition on
the codimensions implies that concatenation yields square matrices
$$ V = \left(V_1, V_2, V_3\right) \qquad \mbox{and} \qquad W = \left(W_1, W_2, W_3\right)$$
satisfying
$$ V^TW = \mbox{id}_{n \times n}.$$
From this, obviously
\[ \gamma:= |\det(V)| = |\det W|^{-1}. \]

We recall that, given a matrix $A$, we use $\left| A\right|$ as an
abbreviation for Gramian determinants, i.e.
$$ \left| A \right| = \left|\mbox{det}\left(A A^t\right)\right|^{1/2},$$
where $A$ may be any $n \times \cdot$ matrix and the number can be understood as the $\cdot-$dimensional Hausdorff measure of the parallelepiped formed by the vectors. 
In particular, in concordance with our previous use of the notation $\gamma_3$ in the fiber representation
$$ \gamma_i := |W_i^t|.$$
For simplicity we understand the indices as integers modulo $3$ below. 

\begin{lemma} \label{faces} For each $i \in \left\{1,2,3\right\}$, we have
\[ |(W_{i-1} W_{i+1})^t| = \gamma^{-1} \] 
and
\[ |(V_{i-1}V_{i+1})^t|= \gamma |W_i^t|\] 
\end{lemma}
\begin{proof} We show a set inclusion about spaces spanned by these vectors:
$$ \mbox{span}(W_i - V_i) \subset \spa(W_{i-1} W_{i+1}).$$
Using that we have dual basis, a vector is in the right-hand set if and only if
its projections onto any vector from $W_i$ is 0, which is precisely the case if its projection
onto any vector from $V_i$ is 0. Now, with duality and orthonormality of elements in $V_i$
$$ \left\langle v_i^a, w_i^b - v_i^b \right\rangle = \left\langle v_i^a, w_i^b \right\rangle - \left\langle v_i^a, v_i^b \right\rangle = \delta_{ab} - \delta_{ab} = 0.$$
Now, using the multilinearity of the determinant, we get
\begin{align*}
 \gamma^{-1} &= \det(W_{i-1} W_{i+1} W_i)  = \det(W_{i-1} W_{i+1} V_i).
\end{align*}
Since we are dealing with dual bases, the volume of the p described by the determinant factors into
\begin{align*}
 \gamma^{-1} &= \det(W_{i-1} W_{i+1} V_i)\\
&= \det( (W_{i-1} W_{i+1})^t (W_{i-1} W_{i+1}))^{1/2} \det(V_i^t V_i)^{1/2}\\
& = |(W_{i-1} W_{i+1})^t| .
\end{align*}
A similar argument implies 
\[ \spa( W_i (W_i W_i^t)^{-1} - V_i) \subset \spa(V_{i-1},V_{i+1}) \] 
and hence 
\[ \gamma = \det(V_1V_2V_3) = \det(V_1V_2 W_3 (W_3 W_3^t)^{-1}) = 
|(V_1V_2)^t| |W_3^t|^{-1} \]
which is the second identity. 
\end{proof}

\subsection{Proof of the linear case.} Using these geometric considerations, we are now able
to produce a complete proof of the linear case. The main idea consists in applying Cauchy-Schwarz
twice, once globally on a surface and once within the fiber of integration. This will lead to a suitable
decoupling of the quantity allowing for a full solution via a further change of coordinates.
\begin{prop} Let $H_1, H_2$ and $H_3$ be as above.
 Then, for all $f_1 \in L^2(H_1), f_2 \in L^2(H_2)$, 
$$  \|f_{1}*f_{2}\|_{L^2(H_3)}  \le  \frac{1}{\sqrt{\gamma}} \Vert f_1 \Vert_{L^2(H_1)} \Vert f_2 \Vert_{L^2(H_2)}.
$$
\end{prop}
\begin{proof} 
Let $V_i$ and $V$ be as above and let $W_i$ be the dual basis to $V_i$. Then 
$H_3 = W_1 \times W_2$ and $|(W_1,W_2)^t|= \gamma$. 
Thus, with $m^d$ the $d$ dimensional Lebesgue measure, by the area formula and
Lemma \ref{faces},  
$$
\label{trans1}  \int_{H^3}  (f_1 * f_2 )^2 d\mathcal{H}^{n_3}  
=  \gamma^{-1}    \int_{\R^{n_3}}  (f_1* f_2 ((W_1,W_2)y )^2 dm^{n_3}(y). \qquad \qquad \mbox{(3.3)}
$$
Moreover, with $s_i \in \spa(W_i)$, by the coarea and area formulas  
\[
\begin{split} 
  f_1 * f_2(s_1+s_2+s_3) = & |(V_1,V_2)^t|^{-1} 
\int_{\spa W_3} f_1(s_2+t) f_2(s_1+t) d\mathcal{H}^{n-n_3}(t)  
\\ = & \frac{|W_3^t|}{|(V_1,V_2)^t|}
\int_{\R^{n-n_3}}  f_1(s_2+ W_3z) f_2(s_1-W_3z) dm^{n-n_3}(z) 
\\ = & \gamma^{-1} \int_{\R^{n-n_3}}  f_1(s_2+ W_3z) f_2(s_1-W_3z) dm^{n-n_3}(z)
\end{split} 
\] 
where we used the second identity of  Lemma \ref{faces}. We continue  
\[ 
\begin{split} 
\Vert f_1  * f_2  \Vert_{L^2(H_3)}^2 
\le & \gamma^{-3}  \int_{\R^{n_3}} 
  \int_{\R^{n-n_3}} f_1^2(W_2 y_2 + W_3 y_3) dm^{n-n_3}(y_3)
\\ &\qquad  \times \int_{\R^{n-n_3} }
  f_2^2(W_1 y_1 + W_3 y_3) dm^{n-n_3}(y_3)  dm^{n_3}(y_1,y_2) 
\\ = & \gamma^{-3} \left( \int_{\R^{n_1}} f_1^2(W_2y_2+W_3y_3)dm^{n_1}(y_2,y_3) \right) \\
&\times \left( \int_{\R^{n_2}} f_2^2(W_1y_1+W_3y_3)dm^{n_2}(y_1,y_3) \right)
\\ = & \gamma^{-1}  \Vert f_1 \Vert_{L^2(H_1)}^2 \Vert f_2 \Vert_{L^2(H_2)}^2
\end{split} 
\] 
where we used again the considerations of \eqref{trans1}. 
 \end{proof}

\section{Convolution estimate: Proof of Theorem 1}

\subsection{Outline.} In this section we extend the previous argument
from hyperplanes to general polyhedral surfaces. We emphasize that the
problem has no special intrinsic connection to polyhedral surfaces and
we use them solely out of convenience: they are well suited for
approximating $C^1-$surfaces and, due to their piecewise linear
nature, allow for a relatively slick reduction to the purely linear
case as localizing will lead us with a locally linear
geometry. Naturally, since we need to be able to carry out a limit
process in the end, all our estimates will be independent of the number of faces
of the polyhedral surfaces.

\subsection{Fiber representations.} There is an explicit expression
for the convolution in terms of integration along the corresponding
fiber, that will be useful in the proof of the statement. We will
explicitely write $f_{i, \mu_i}$ to highlight the importance of the
surface along with function value in the following impression. For a
fixed $z \in \Sigma_{3}$, if
$$ \Gamma_z = \left\{x \in \Sigma_{1}: z-x \in \Sigma_{2}\right\} = \Sigma_1 \cap (z-\Sigma_2),$$
then, by thickening the surfaces and using the coarea formula,
\[  (f_{1,\mu_1} * f_{2,\mu_2}) (z) = \int_{x \in \Gamma_z}  \gamma_3^{-1}(x,y)     f_1(x)f_2(z-x) d\mathcal{H}^{n-n_3}(x) \]  
where the Gramian determinant $\gamma_3$ is given by
$$  \gamma_3(x,y) = |\det ( (N_1,N_2)^t (N_1,N_2))|^{\frac12}.\label{bilinerarconv} $$
It is the $(n-n_3)$-dimensional volume
of the parallelotope formed by the normal vectors of $\Sigma_1$ and
$\Sigma_2$. The geometric quantity introduced is implied by the
following identity for the Gramian determinant in $(n-n_3)$
dimensional space. This can be compared with affine-invariant
formulation of Bennett \& Bez' Brascamp-Lieb inequality in terms of
exterior algebra, where similar expression play comparable roles. 

\subsection{Polyhedral surfaces.} Using the fiber representations, we are now able to deal with the general
case of polyhedral surfaces -- the main idea of the proof is a suitable application of Cauchy-Schwarz on
two different domains, which allows for a suitable decoupling to take place and gives rise to a much simpler linear
expression.

\begin{prop} \label{pro}
Suppose the $\Sigma_i$ are polyhedral surfaces and that $$\gamma(x,y,z) \ge \gamma_0,$$ whenever all three normal vectors are defined.
 Then, for all $f_1 \in L^2(\Sigma_1), f_2 \in L^2(\Sigma_2)$, 
$$  \|f_{1}*f_{2}\|_{L^2(\Sigma_{3})}  \le  \gamma_0^{-\frac{3}{2}} \Vert f_1 \Vert_{L^2(\Sigma_1)} \Vert f_2 \Vert_{L^2(\Sigma_2)}.$$
\end{prop}
\begin{proof} We assume without loss of generality that the polyhedral surfaces are made up of finitely many faces and will prove a bound uniform in the number of faces. The claim follows from the inequality
\begin{align*}
I := &\int_{\Sigma_3}  \left(\gamma_0 \int_{\Gamma_z}  \frac{\gamma(x,z-x,z)^{\frac12}}{\gamma_3(x,z-x)} f_1(x)f_2(z-x) d\mathcal{H}^{n_1+n_2-n}(x) \right)^2   d\mathcal{H}^{n_3}(z) \\
 &\leq \|f_1\|^2_{L^2(\Sigma_1)} \|f_2\|^2_{L^2(\Sigma_2)}.
\end{align*}
which we  will prove now: applying Cauchy-Schwarz inequality  in the fiber $\Gamma_z$  yields
\[\begin{split}  I \leq & \int_{\Sigma_3}   \left(\gamma_0\int_{\Gamma_z}
\frac{\gamma(x,z-x,z)^{\frac12}}{\gamma_3(x,z-x)}
f_1(x)^2d\mathcal{H}^{n_1+n_2-n}(x)\right)
\\ & \times    \left(\gamma_0\int_{\Gamma_z}
\frac{\gamma(x',z-x',z)^{\frac12}}{\gamma_3(x',z-x')}
f_2(z-x')^2d\mathcal{H}^{n_1+n_2-n}(x')\right) \, d\mathcal{H}^{n_3}(z).
\end{split} 
\] 
The right hand side  is \textit{linear} in $f_i^2$  which we exploit by  
localization procedure in $\Sigma_1$ and $\Sigma_2$. Take a decomposition of $\Sigma_1, \Sigma_2$
$$ \Sigma_1 = \dot \bigcup_{j}{\Sigma_{1,j}} \qquad   \Sigma_2 = \dot \bigcup_{k}{\Sigma_{2,k}} \qquad 
\Sigma_3 = \dot \bigcup_{l}{\Sigma_{3,l}}
$$
with the property that for each $j$ the normal vectors $\nu_1$ and $\nu_2$ are constant on the sets $ \Sigma_{1,j}$ and $\Sigma_{2,k})$. By a further possibly countable decomposition 
we may also achieve that $(\Sigma_{1,j}+\Sigma_{2,k})\cap \Sigma_3$ lies in a single  set $\Sigma_{3,l}$. 
We abbreviate $ f_{i,\cdot} := f_{i}\chi_{\Sigma_{i,\cdot}}$. Using this
decomposition, it remains to estimate
\[\begin{split}  \int_{\Sigma_3}  \left(\gamma_0\int_{\Gamma_z}
\frac{\gamma(x,z-x,z)^{\frac12}}{\gamma_3(x,z-x)}
\sum_{j_1} f_{1,{j_2}}(x)^2d\mathcal{H}^{n_1+n_2-n}(x)\right)
\hspace{-7cm} & \\ & \times    \left(\gamma_0\int_{\Gamma_z}
\frac{\gamma(x',z-x',z)^{\frac12}}{\gamma_3(x',z-x')}
\sum_{j_2}f_{2,{j_2}} (z-x')^2d\mathcal{H}^{n_1+n_2-n}(x')\right) \, d\mathcal{H}^{n_3}(z).
\end{split} 
\] 

Note that $f_{i,i_1}$ and $f_{i,i_2}$ have disjoint support unless $i_1 = i_2$. Thus, we can expand
the square and use the linearity of the integral  to see that is suffices to estimate 
\[\begin{split}  \int_{\Sigma_3}  \left(\gamma_0 \int_{\Gamma_z}
\frac{\gamma(x,z-x,z)^{\frac12}}{\gamma_3(x,z-x)}
f_{1,{j_2}}(x)^2d\mathcal{H}^{n_1+n_2-n}(x)\right)
\hspace{-7cm} & \\ & \times   \left(\gamma_0 \int_{\Gamma_z}
\frac{\gamma(x',z-x',z)^{\frac12}}{\gamma_3(x',z-x')}
f_{2,{j_2}} (z-x')^2d\mathcal{H}^{n_1+n_2-n}(x')\right) \, d\mathcal{H}^{n_3}(z)
\\  \le & \Vert  f_1 \Vert_{L^2(\Sigma_1)}^2  \Vert  f_2 \Vert_{L^2(\Sigma_1)}^2.
\end{split} 
\] 
Note that the geometric expression
$\gamma(x,z-x,z)^{\frac12}/\gamma_3(x,z-x)$ is constant for every
$j_1, j_2$ provided we have chosen set with sufficiently small
support, because of the polyhedral natural of the surfaces and the
choice of our decomposition. It is evident that $\gamma_3(\cdot, \cdot) \geq \gamma_0$
and we may thus estimate the expression from above by
\[\begin{split}  J = \int_{\Sigma_3}   \left(\int_{\Gamma_z}
\gamma(x,z-x,z)^{\frac12}
f_{1,j_1}(x)^2d\mathcal{H}^{n_1+n_2-n}(x)\right)
\hspace{-7cm} & \\ & \times    \left(\int_{\Gamma_z}
\gamma(x,z-x',z)^{\frac12}
f_{2,{j_2}} (z-x')^2d\mathcal{H}^{n_1+n_2-n}(x')\right) \, d\mathcal{H}^{n_3}(z)
\end{split} 
\] 
For every fixed $j_1, j_2$, this is now
precisely the expression we had to deal with in our proof in the
linear case -- redoing the same steps as before yields that for every
$j_1, j_2$ 
$$  J  \le  \Vert  f_1 \Vert_{L^2(\Sigma_1)}^2  \Vert  f_2 \Vert_{L^2(\Sigma_1)}^2$$
and this concludes the proof.
\end{proof}

\begin{proof}[Proposition \ref{pro} implies Theorem 1.]
 Recall that we understand the restriction $f_{1}*f_{2}\big|_{\Sigma_{3}}$ in the sense of 
the dense embedding $C_0^{\infty}(\mathbb{R}^3) \hookrightarrow L^2(\mathbb{R}^3)$. It suffices to consider continuous functions $f_i$ with compact support defined on $\mathbb{R}^n$. But then both sides converge for $C^1$ hypersurfaces as the polyhedral approximation tends to the hypersurface.
\end{proof}

\section{Proof of the Linear Brascamp-Lieb inequality}
This section gives a new proof for the linear Brascamp-Lieb inequality much in the same spirit
as the proof for the convolution estimate in the linear case. Again, the proof will be stable enough
to allow it being transferred to the nonlinear setting.
\begin{prop}
Let $A_i: \R^n \to \R^{n_i}$ be linear maps with maximal rank $n_i$ for $1 \leq i \leq 3$, where
$$n_1+n_2+n_3=2 n.$$
Let $V_i$ be an orthonomal basis of the null space of $A_i$, let $V = (v_1, v_2, v_3)$ and
\[ \gamma = |\det V|. \]
Then we have for all $f_i \in L^2(\R^{n_i})$
\[ \prod_{i=1}^3 |A_i|^{1/2} \int_{\R^n} \prod_{i=1}^3 f_i \circ A_i(x) dx
\le   \gamma^{-1/2}  \prod \Vert f_i \Vert_{L^2(\R^{n_i})}. \]
\end{prop}

\begin{proof}
Let $P$ be the parallelepiped formed by $V$. Its volume is $\gamma$.
The matrix $A_1$ maps $P$ to a parallelepiped $P_1$ in $\R^{n_1}$ which is
spanned be the image of the vectors of $V_2$ and $V_3$.
We claim that its  volume is
\begin{equation} |(W_2W_3)^t|^{-1} |A_1|= \gamma |A_1|.  \end{equation}
First we reduce the assertion to the case $ A_1A_1^T = 1_{\R^{n_1}}$
which we can achieve by a linear change of coordinates in
$\R^{n_1}$. The volume of $P$ is the same as the volume of the
parallelepiped which we obtain by orthogonally projecting the vectors
$v_2^a$ and $v_3^a$ along $V_1$. Then we obtain $\tilde w_i^{a}$ such
that $(V, \tilde W^1, \tilde W^2)$ is the dual  basis to $(V,
W^1,W^2)$ and the $n_1$ dimensional volume of the parallelepiped
spanned by the columns of $\tilde W_2$ and $\tilde W_3$ is $\gamma$.
We repeat this argument for the other variables and   obtain
 \[ \int \prod \chi_{P_i} \circ A_i(x) dx = \gamma \]
as well as
\[ \Vert \chi_{P_1} \Vert_{L^2(\R^{n_1})} =  |W_{2}W_{3}|^{-1/2}
|A_1|^{1/2} = \gamma^{1/2} |A_1|^{1/2}. \]
This implies the claimed formula with equality for characteristic functions 
of such parallepipeds.  
For general  functions we proceed differently and we apply the coarea formula and Cauchy-Schwarz
inequality twice: 
\[
\begin{split} 
\prod_{i=1}^3 |A_i|^{\frac12} \int \prod_{i=1}^3 f_i \circ A_i(x) dx
\hspace{-3cm} & \\
= &  |A_3|^{-\frac12} |A_1|^{\frac12} |A_2|^{\frac12} \int_{\R^{n_3}} f_3(z) \int_{\{x: A_3(x)=z\} }    f_1(A_1x)f_2(A_2x) d\mathcal{H}^{n-n_3}dm^{n_3}(z) \\ 
 \le &    \Vert f_3 \Vert_{L^2(\R^{n_3})}  \int_{\R^{n_3}} \prod_{i=1}^{2}\left( \int_{\{x: A_3x = z\}} |A_3|^{\frac12} |A_i|^{-1} |f_i(A_i(x))|^2 d\mathcal{H}^{n-n_3}\right) dm^{n_3}(z)
\end{split} 
\]
We check the validity of the desired estimate for the special case of characteristic functions 
of parallepipeds $f_i$ for $i=1,2$ by plugging them in the right hand side expression. 
In this case, the integral over the fiber gives $1$
and we have to integrate the very same parellepiped in $\R^{n_3}$ as above and the 
squares of the $L^2$ norms are $|A_i|^{-1} \times \gamma$. The argument follows
for general functions by either interpreting these consideration as 
a determination of Gramian determinants in area and coarea formulas, 
or, alternatively, by approximating general continuous functions 
by  sums of multiples of characteristic functions of parallelepipeds.

\end{proof}

\section{Nonlinear Brascamp-Lieb inequality: Proof of Theorem 2}
This section concludes with a proof of the Brascamp-Lieb inequality
 $$  \int_{\R^n}  \left(\prod_{i=1}^3 |D\phi_i| \right)^{\frac12}
 f_1(\phi_1(x)) f_2(\phi_2(x)) f_3(\phi_3(x)) dx \le 
\sqrt{\frac{\rho}{\gamma_0}} \Vert f_1 \Vert_{L^2(\R^{n_1})} \Vert f_2 \Vert_{L^2(\R^{n_2})} \Vert f_3 \Vert_{L^2(\R^{n_3})}.$$
The argument is essentially identical to our previous argument and merely phrased in a slightly different language. Previously we were dealing with
the linear structure $x+y=z$ in $\mathbb{R}^n$ as induced by the convolution and nonlinear (but transversal) hypersurfaces. Now we are dealing with
flat surfaces and a nonlinear (but transversal) fiber structure. The crucial idea is, once again, that applying
Cauchy-Schwarz once on $\mathbb{R}^{n_3}$ and then once in the integration fiber yields a bilinear 
expression which, on small scales, reduces to the linear case while $L^2-$orthogonality allows for a reduction
to small scales.

\begin{proof}[Proof of Theorem 2.] It suffices to consider nonnegative functions; this implies that the integrals are defined -- possibly with the value $\infty$.  We start by rewriting the squared expression by the coarea formula as
$$ \left(\int_{\mathbb{R}^{n_3}} f_3(z)  \left(
\int_{\left\{x: \phi_3(x) = z\right\}}  |D\phi_1|^{\frac12}  |D\phi_2|^{\frac12}|D\phi_3|^{-\frac12} f_1(\phi_1(x))f_2(\phi_2(x))d\mathcal{H}^{n-n_3}(x) \right)
dm^{n_3}(z)\right)^2$$
The next step is again $L^2-$duality: we apply the Cauchy-Schwarz inequality on $\mathbb{R}^{n_3}$ and eliminate the $f_3$ term entirely. We rewrite the condition $n_1+n_2+n_3 = 2n$ as 
$$ \frac{n-n_2}{n_3} + \frac{n-n_1}{n_3} = 1 \qquad \mbox{and thus} \qquad 
|D\phi_3|^{-\frac12} = |D\phi_3|^{-\frac{n-n_2}{2n_3}}|D\phi_3|^{-\frac{n-n_1}{2n_3}}$$
and use the Cauchy-Schwarz inequality once more in
the fiber, which results in the integral
\begin{multline*}
\int_{\R^{n_3}}  \left(\int_{\left\{x: \phi_3(x) = z\right\}} |D\phi_1||D\phi_3|^{-\frac{n-n_2}{n_3} } 
 f_1(\phi_1(x))^2 d\mathcal{H}^{n-n_3}(x) \right) \\
\times \left(\int_{\left\{x: \phi_3(x) = z\right\}} |D\phi_2| |D\phi_3|^{-\frac{n-n_1}{n_3}}f_2(\phi_2(x))^2d\mathcal{H}^{n-n_3}(x)\right) dm^{n_3}(z) .
\end{multline*}
It is crucial to be aware of the arising dimensions: the total integral is an integral in
dimension $2(n-n_3)+n_3 = n_1+n_2$, and we want to bound it in terms
of the square of the $L^2$ norms, which again is a related to an
integral over a set of dimension $n_1+n_2$. The transversality condition implies that there
is a bijective mapping between the sets $\Sigma_1\times \Sigma_2$ and 
\[ \{ (x,y) \in \R^{n} \times \R^n : \phi_3(x)=\phi_3(y) \}. \]
Instead of constructing and working with this map directly we choose a more geometric
and less technical approach: it suffices again to verify the estimate for functions $f_1$ 
and $f_2$ supported on small parallelepipeds (which is implicitly a construction of 
the map between the two spaces). Indeed, if 
we can prove the inequality for characteristic functions
$$ (f_1, f_2) = (\chi_{E}, \chi_{F}) $$
for small parallelepipeds $(E,F) \subset \R^{n_1}\times \R^{n_2}$, the
entire inequality follows from the bilinearity of the expression and
the $L^2-$orthogonality of functions with disjoint support by mere
addition. As in the case of the convolution, it suffices to consider
piecewise linear maps $\phi_i$.  Now let $x,y \in \R^n$ such that
$\phi_3(x)= \phi_3(y) $. We may restrict ourselves to linear maps
$A_1=\phi_1$ and $A_2=\phi_2$ as well as $\phi_3= A_3^x$ and
$\phi_3=A_3^y$ near $x$ resp. $y$.  The two linear maps will differ in
general. There is no harm in applying a rotation $O$ at $x$. Hence we
may assume that the null spaces of $A_3^x$ and $A_3^y$ are the
same. We proceed as in the linear situation for which all quantities
have been explicitely computed, with a linear map $A_3$ defined by the
null space, by $A_3=A_3^x$ on the null space   of $A_1$ and by
$A_3=A_3^y$ on the null space   of $A_2$.

\medskip

The only difference to the previous case concerns the third map. The Gramian determinant 
is given by the volume of the parallelepiped spanned by image of $V_1 $ under the map
$A_3^x$ and $V_2$ under the map $A_3^y$, respectively. This volume is certainly biggest
if these maps induce an orthogonal image, in which case it is bounded by the product 
of the $(n-n_1)$ dimensional volume of the image of $V^1$ under $A_3^x$ and the
$(n-n_2)$ dimensional volume of the image of $V^2$ under
$A_3^y$. Comparison with the linear case shows that this factor is
controlled by the product of
$$ \rho_{1}(x) =    |D\phi_3(x)|^{-\frac{n-n_2}{n_3}}  \prod_{j=n-n_2+1}^{n_3} \sigma_j 
\quad \mbox{and} \quad  \rho_{2}(y)=  |D\phi_3(x)|^{-\frac{n-n_1}{n_3}} \prod_{j=n-n_1+1}^{n_3} \sigma_j. $$\end{proof}

\textbf{Acknowledgments.} We are grateful to Sebastian Herr for valuable discussions. The second
author was supported by a Hausdorff scholarship of the Bonn International Graduate School and
SFB 1060 of the DFG.

\end{document}